\documentclass[11pt]{amsart}
\usepackage{amsmath, amssymb, amsthm}
\usepackage{enumerate}
\newcommand{\R}{{\mathbb{R}}}

\newtheorem*{theorema}{Theorem}
\newtheorem*{theoremb}{Example 1}
\newtheorem*{theoremc}{Example 2}
\newtheorem*{theoremd}{Example 3}
\newtheorem*{theoreme}{Example 4}

\newcommand{\taf}{{\hskip 5pt} $\blacksquare$
                  \renewcommand{\qedsymbol}{}}

\begin{document}
\title{Modified Riemann Sums of Riemann-Stieltjes Integrable Functions}
\author{Alberto Torchinsky}
\date{}
\date{}
\address{Department of Mathematics\\ Indiana University\\
Bloomington IN 47405} \email{torchins@indiana.edu}

\begin{abstract} In this paper, motivated by physical considerations,
 we introduce the notion of modified Riemann 
sums of Riemann-Stieltjes integrable functions,  show that they converge,
and  compute them explicitely under various assumptions.
\end{abstract}

\maketitle

\section*{Introduction}
This note concerns the notion of     modified Riemann sums of Riemann-Stieltjes integrable  functions. This basic concept, motivated by an experiment in signal retrieval \cite{DHT},  though it seems that it should be  known, is not present in the classic or standard literature in the area. 

We  begin by describing a simplified version of the physical situation  in EE terms. One of the most common means of information transmission is amplitude modulation (AM), where a signal of frequency $f$ is multiplied by a carrier wave at a much higher frequency $f_c$, permitting signal transmission over long distances. 
Generally, the receiver recovers the originally modulated signal by multiplying the received waveform by a ``local oscillator,'' i.e., a local copy of the carrier wave at $f_c$, and then filtering away the high frequency components, forming the basis of AM radio. Occasionally, however, direct analysis of the signal itself is required rather than first demodulating the carrier wave, e.g., when a local oscillator field is not available. In these cases, it becomes necessary to develop an accurate treatment of the original signal itself. 

When the signals are modulated by a carrier square wave of unit amplitude, as may occur in fast digitization of modulated signals \cite{DHT}, two considerations may come into play. First, of course, Gibbs' phenomenon \cite{GP}. Second, the realization that the operations on the original signal correspond  to the integration of a continuous function $f$ over a domain where the support of  $f$ is periodically reduced by a factor of 1/2. In terms of the approximating Riemann sums of the integral of $f$, this translates into halving each interval that appears in the  Riemann sums  $\sum_{k=1}^n f(x_k)|I_k|$ of  $f$,   and considering modified Riemann sums of the form $\sum_{k=1}^n f(x_k^1)|I_k^1|$, where each $x_k^1$ lies in the  left half $I_k^1$ of $I_k, 1 \le k \le n$.  That these sums converge to $(1/2)\int_I f$ was derived experimentally in \cite{DHT}.

\vskip .065in
To cast this observation in a general setting  we begin by introducing some definitions and notations. 
Fix a closed finite interval  $I=[a,b]\subset \R$, and let $\Psi$ be an increasing function defined   on $I$. 
For a partition $\mathcal P=\{I_k\}$ of $I$, where $I_k=[x_{k,l},x_{k,r}]$, and a bounded function $f$  on $I$, let  
 $U(f,\Psi, \mathcal{P})$ and  $L(f,\Psi, \mathcal{P})$ denote  the  upper and lower Riemann sums of $f$ with respect to $\Psi$ along $\mathcal{P}$ on $I$,  i.e.,
\[U(f,\Psi, \mathcal{P})=\sum_k \big(\sup_{I_k}f\big)\, \big(\Psi(x_{k,r})- \Psi(x_{k,l})\big),\]
and  
\[ L(f,\Psi, \mathcal{P})=\sum_k \big(\inf_{I_k}f\big)\, \big(\Psi(x_{k,r})- \Psi(x_{k,l})\big),
\]
respectively, and set
 \[ U(f,\Psi)=\inf_{\mathcal{P}}\, U(f,\Psi,\mathcal{P}),\quad{\rm{and}} \quad
 L(f,\Psi)=\sup_{\mathcal{P}} L(f,\Psi, \mathcal{P}).
 \]
 
We say that $f$ is Riemann integrable with respect to $\Psi$ on $I$ if $U(f,\Psi)=L(f,\Psi)$, and in this case the common value is denoted $\int_I f\,d\Psi$, the Riemann integral of $f$ with respect to $\Psi$ on $I$. 

When $\Psi(x)=x$ one gets the usual Riemann integral on $I$, and  $\Psi$ is omitted in the above notations. And, throughout this note, when it is clear from the context, integrable means Riemann integrable with respect to $\Psi(x)=x$, and Riemann-Stieltjes integrable means integrable with respect to a   general $\Psi$.

 Then, the Riemann-Stieltjes integrability of $f$ is equivalent to the fact that,  given $\varepsilon >0$, there is a partition ${\mathcal{Q}}$ of $I$, which may depend on $\varepsilon$,  such that 
\begin{equation} 0\le U(f,\Psi,{\mathcal{Q}})-L(f,\Psi, {\mathcal {Q}})\le \varepsilon.
\end{equation}

In turn,  (1) is equivalent to the existence of a sequence of partitions  $\{\mathcal {Q}_n\}$ of $I$ such that, 
\begin{equation*}
\lim_{n}  \big( U(f,\Psi,\mathcal {Q}_n) - L(f,\Psi, \mathcal {Q}_n) \big) = 0,
\end{equation*}
and  in this case  
\begin{equation} \lim_{n} U(f,\Psi, \mathcal Q_n) = \lim_{n}L(f, \Psi, {\mathcal Q}_n)=
\int_I f\,d\Psi.
\end{equation}

Moreover, the intervals in the partitions ${\mathcal{Q}}$ in  (1) and  $\{{\mathcal Q}_n\}$ in (2) may be chosen to be arbitrarily small. To verify the latter assertion, let  the partition ${\mathcal T}_n$ of $I$ be obtained by dividing $I$ into $n$ equal intervals   and consider  the partition 
 $\mathcal P_n$ of $I$ which is the common refinement of ${\mathcal Q}_n$ and ${\mathcal T}_n$. Then,  each interval  in $ \mathcal P_n$ has length less than or equal to $|I|/n$, and, as  is
readily seen,
\begin{equation}
\lim_{n}  \big( U(f,\Psi, \mathcal { P}_n) - L(f,\Psi,  \mathcal {P}_n) \big) = 0\,,
\end{equation}
and
\begin{equation*}\lim_n  U(f,\Psi,  \mathcal {P}_n)  =
\lim_{n}   L(f,\Psi,  \mathcal {P}_n)=\int_I f\,d\Psi. 
\end{equation*}

Clearly in  this case the sequence $\{ S(f,\Psi, \mathcal{P}_n)\}$ of arbitrary Riemann sums of $f$ with respect to $\Psi$ on $I$ corresponding to the partitions $\{\mathcal{P}_n\}$, consisting of the intervals $\mathcal {P}_n=\{I_k^n\}$, with $I_k^n=[x_{k,l}^n,x_{k,r}^n]$,  given by 
\begin{equation*} S(f,\Psi, \mathcal{P}_n) =\sum_{k} f(x_k^n)\, \big(\Psi(x_{k,r}^n)- \Psi(x_{k,l}^n)\big)\,,\quad x_k^n\in I_k^n,
\end{equation*}
which lie between $L(f,\Psi,\mathcal{P}_n)$ and $ U(f,\Psi,\mathcal{P}_n)$, will also converge to the common limit above, i.e.,
\begin{equation} \lim_n S(f,\Psi, \mathcal{P}_n)=\int_I f d\Psi.
\end{equation}

Integrability can also be characterized in terms of the oscillation of a function  \cite{Bruckner}, \cite{Hunter}, \cite{Torch2}. Recall that, given a bounded function $f$ defined on $I$  and an interval $J\subset I$, the oscillation 
${\rm {osc\, }}(f, J)$ of $f$ on $J$ is defined as 
$ {\rm {osc\, }}(f, J) = \sup_J f - \inf_J f$. Then, a bounded function $f$ is Riemann integrable with respect to $\Psi$ on $I$ iff,  given $\varepsilon >0$, there is a partition ${\mathcal{P}}=\{I_k\}$ of $I$, which 
may  depend on $\varepsilon$,  such that 
\begin{equation}  \sum_k {\rm {osc\, }}( f, I_k)\,\big(\Psi(x_{k,r})- \Psi(x_{k,l})\big)\le\varepsilon.
\end{equation}

And, a sequential characterization  holds, to wit,  (5) is equivalent to the existence of a sequence of partitions  $\{\mathcal {P}_n\}$ of $I$ consisting of the intervals $\mathcal {P}_n=\{I_k^n\}$, with $I_k^n=[x_{k,l}^n,x_{k,r}^n]$,  such that 
\begin{equation} \lim_n \sum_k {\rm {osc\, }}( f, I_k^n)\,\big(\Psi(x_{k,r}^n)- \Psi(x_{k,l}^n)\big)=0\,.
\end{equation}

Moreover, the partitions ${\mathcal{P}}$ in  (5) and  $\{{\mathcal P}_n\}$ in (6) may be refined preserving the above relations.

\section*{Modified Riemann Sums.}

We will henceforth assume  that $\Psi$ is an indefinite integral on $I$, i.e., for a  bounded, positive Riemann integrable function $\psi$ defined on  $I$, we have 
\begin{equation*} \Psi(x)=\Psi(a)+ \int_{[a,x]} \psi\,,\quad  x\in I.
\end{equation*}

Let now $\Phi$ be a set mapping defined on the subintervals $J$ of $I$ that assigns to each  $J\subset I$ a subset $J^1=\Phi(J)\subset I$ with the following properties: 

(i) $\chi_{\Phi(J)}$ is Riemann integrable, and so, the $\Psi$--length $|J^1|_\Psi$ of $J^1$ is well-defined as $|J^1|_\Psi=0$ if $J^1=\emptyset$, and  $|J^1|_\Psi=\int_I \chi_{\Phi(J)}\psi $, otherwise, and, 

(ii) there exists $\eta>0$ such that $J^1=\Phi(J)\subset J$, whenever $|J|<\eta$.


 
\vskip .1in

Given a partition $\mathcal{P}$ of $I$ consisting of intervals $I_1, \ldots$, $I_m$,  we will denote with $\mathcal{P}^1$ the collection of subsets of $I$ consisting of  $I_1^1=\Phi(I_1), \ldots, I_m^1 =\Phi(I_m)$, say, and set
\[ u(f,\Psi, \mathcal{P}^1) =\sum_{k}\big(\sup_{I_k^1}f\big)\, |I^1_k|_\Psi,\quad {\rm {and}} \quad 
 l(f,\Psi,\mathcal{P}^1) =\sum_{k} \big(\inf_{I_k^1}f\big)\, |I^1_k|_\Psi.
\]
 We call these expressions the {\it modified upper and lower Riemann sums of $f$ with respect to $\Psi$  along $\mathcal{P}$ on $I$}, respectively, and set 
\[u(f,\Psi)= \inf_{{\mathcal P}}  u(f,\Psi,\mathcal{P}^1),\quad{\rm{and}}\quad  l(f,\Psi) = \sup_{\mathcal  {P}}\,  l(f,\Psi,\mathcal{P}^1)\,.
\]

Note that $ l(f,\Psi) \le  u(f,\Psi).$  We  say that {\it the modified Riemann sums of $f$ with respect to $\Psi$ on $I$ converge} if
\[u(f,\Psi)= l(f,\Psi).
\]
In  this case the {\it (arbitrary) modified Riemann sums of $f$ with respect to $\Psi$ on $I$} given by 
\begin{equation} s(f,\Psi, \mathcal{P}^1) =\sum_{k} f(x_k^1)\, |I^1_k|_\Psi\,,\quad x_k^1\in I^1_k,
\end{equation}
which lie between $l(f,\Psi,\mathcal{P}^1)$ and $ u(f,\Psi,\mathcal{P}^1)$, will also converge to the common limit above in a sense that will be made precise by the Theorem below. This theorem  
addresses the convergence of the  modified Riemann sums of a   Riemann-Stieltjes  integrable function, and provides the means to compute the value of the limit. 
\begin{theorema}
Let $f$ be a bounded, Riemann integrable function with respect to $\Psi$  on $I$. Then, the  modified Riemann sums of $f$ with respect to $\Psi$ on $I$ converge.  Moreover,  there is a sequence of partitions  $\{\mathcal {P}_n\}$  of $I$ such that
\begin{equation}
\lim_{n} \big( u(f,\Psi,\mathcal { P}_n^1) - l(f,\Psi, \mathcal {P}_n^1)\big) = 0\,,
\end{equation}
and, for any  sequence of partitions  $\{\mathcal {P}_n\}$  of $I$  that satisfies (8)  we have
\[u(f,\Psi)= \lim_{n} u(f,\Psi, \mathcal {P}_n^1) = \lim_{n} l(f,\Psi, {P}_n^1)=
l(f,\Psi).\] 

Furthermore, the  arbitrary modified Riemann sums $ s(f,\Psi,\mathcal{P}_n^1)$ of $f$ with respect to $\Psi$  on $I$, defined by (7) above, also converge to $u(f,\Psi)=l(f, \Psi)$.
\end{theorema}
\begin{proof}
By (1),  given $\varepsilon >0$, pick a partition $\mathcal{P}$ of $I$ consisting of intervals $I_1,\dots,I_m$,  such that $|I_k|\le |I|/n\le \eta$ for $1\le k\le m$, and 
$  U(f,\Psi,{\mathcal{P}})-L(f,\Psi,{\mathcal {P}})\le \varepsilon\,.$
 Then,  since $|I_k|\le \eta$, by (ii) above, $I_k^1=\Phi(I_k)\subset I_k=[x_{k,l}, x_{k,r}]$ for all $k$, and, therefore, since 
 \[|I_k^1|_\Psi \le |I_k|_\Psi =\big(\Psi(x_{k,r})-\Psi(x_{k,l})\big),\quad {{\rm for}}\ 1\le k\le m,\] 
 it follows that
\begin{equation*} \big(\sup_{I_k^1}f\big)\,|I^1_k|_\Psi - \big( \inf_{I_k^1}f\big )\,|I^1_k|_\Psi  \le 
\big( \sup_{I_k}f- \inf_{I_k}f \big)\,\big(\Psi(x_{k,r})-\Psi(x_{k,l})\big).
\end{equation*}
Thus summing over $k$  we get
\begin{equation}
 u(f,\Psi,{\mathcal{P}^1})-l(f,\Psi, {\mathcal {P}^1})\le    U(f,\Psi,{\mathcal{P}})-L(f,\Psi,{\mathcal {P}}) \le \varepsilon,
\end{equation}
and,  since 
$l(f,\Psi, {\mathcal {P}^1})\le l(f,\Psi)\le u(f,\Psi)\le  u(f,\Psi,{\mathcal{P}^1})$,  (9) yields  
\[ u(f,\Psi)-l(f,\Psi)\le u(f,\Psi,{\mathcal{P}^1})-l(f,\Psi,{\mathcal {P}^1})\le \varepsilon,
\]
which, since $\varepsilon$ is arbitrary, gives the convergence of the modified Riemann sums of $f$ with respect to $\Psi$ on $I$.

Next, since $f$ is Riemann-Stieltjes integrable on $I$, there is a sequence of partitions $\{ \mathcal {P}_n\}$  of $I$ that satisfies (3). Pick $N$ so that $|I|/N\le \eta$, and note that for each $I_\ell\in \mathcal P_n$ with $n\ge N$, we have $|I_\ell|\le \eta$ and so $\Phi(I_\ell)\subset I_\ell$,  and, consequently,  as in (9) it follows that
\begin{equation}
 u(f,\Psi,{ \mathcal{P}_n^1})- l(f,\Psi,{\mathcal {P}_n^1})\le  U(f,\Psi,{\mathcal{P}_n})-L(f,\Psi,{\mathcal {P}_n}),\quad {\rm{all\ }} n\ge N, 
 \end{equation}
and since the right-hand side of (10)   tends to $0$, so does   the left-hand side, and  (8) holds.

Furthermore, since for all $n$,  $l(f,\Psi, { \mathcal{P}_n^1}) \le l(f,\Psi)\le u(f,\Psi) \le u(f,\Psi, { \mathcal{P}_n^1})$,  we have
\[ 0\le \max\big(u(f, \Psi, { \mathcal{P}_n^1}) - u(f,\Psi), l(f,\Psi) -l(f,\Psi,{ \mathcal{P}_n^1})\big)\]
\[\le u(f,\Psi, { \mathcal{P}_n^1})- l(f,\Psi,  { \mathcal{P}_n^1})\,,
\]
and since the right-hand side above tends to $0$, so does the left-hand side, and  therefore
\[ \lim_{n} u(f,\Psi, {\mathcal{P}_n^1} ) = u(f,\Psi),\quad {{\rm and}}\quad   \lim_{n} l(f,\Psi, {\mathcal{P}_n^1} ) = l(f,\Psi).
\]
Then, by (8),
\[ u(f,\Psi)=\lim_{n} u(f,\Psi, {\mathcal{P}_n^1} ) = \lim_{n} l(f,\Psi,{ \mathcal{P}_n^1})  = l(f,\Psi).\]

Finally, since for an  arbitrary modified Riemann sum  $ s(f,\Psi,\mathcal{P}_n^1)$ of $f$ we have   $ l(f,\Psi,\mathcal{P}^1_n)\le s(f,\Psi,\mathcal{P}^1_n)\le u(f,\Psi,\mathcal{P}^1_n),
$ it follows that 
\begin{equation}
 \lim_n s(f,\Psi,\mathcal{P}^1_n)=u(f,\Psi),
 \end{equation}
 and we have finished.
\taf\end{proof}

\section*{Applications.}
We will illustrate two instances of the Theorem, to wit,  when $|J^1|$ is a function of $|J|$ --  linear  for the example that motivated our results, and in general non-linear --, and when $\Phi(J)$ depends on the location of $J$.

We begin with the former case,

\begin{theoremb} With $\alpha$ a fixed constant, $0\le \alpha<1$,  consider a non-negative function  $\phi$ defined on $[\,\alpha,\alpha +|I|\,]$  that satisfies: (i) $\phi(t)\le t-\alpha$, for $\alpha\le t< \beta$, for some constant $\beta < \alpha+ |I|$, (ii) $\phi(\alpha)=0$, and, (iii) $\phi$ is right-differentiable at $t=\alpha$, i.e.,  given $\varepsilon>0$, there exists $\delta>0$ so that for $\phi_+'(a)$,
\begin{equation}
\Big|\, \frac{\phi(\alpha +t)}{t}  - \phi_+'(\alpha)\, \Big| \le \varepsilon, \quad {\rm {whenever }}\ 0<t<\delta.
\end{equation}

Let the mapping $\Phi$ be defined on the subintervals $J$ of $I$ as follows:  $\Phi(J)=J^1$, where 
 $\chi_{\Phi(J)}$ is Riemann integrable, and $|J^1|=\phi (\alpha+|J|)$. 
 
Then, if $f$ is integrable with respect to $\Psi$ on $I$,  $f\psi$ is integrable on $I$, and 
\[ u(f,\Psi)=\phi_+'(\alpha) \int_I f\psi.
\] 
\end{theoremb}
 
First, observe that $\Phi$ is well-defined. Indeed, for $|J|\le \beta-\alpha$ we have  $|J^1|\le (\alpha +|J|)-\alpha=|J|$, and  so we may, and do, pick  $J^1=\Phi(J)\subset J$.  

Now,  $f$ is integrable with respect to $\Psi$ on $I$ iff  $f\psi$ is integrable on $I$, and $\int_I f\,d\Psi=\int_I f\,\psi$, \cite{Lopez1}, \cite{Torch1}. The proof of the necessity, which is what is needed here, is essentially contained in the discussion below for $\phi(t)=t$ and $\Phi(J)=J$, all $J$.

Let the sequence   of partitions $\{{\mathcal P}_n\}$  of $I$ satisfy simultaneously (3), hence also (11), for $f$ with respect to $\Psi$,  and (3),  hence also (4), for $f\psi$,   and  (6) for $\psi$. Let $N$ be such that $|I|/N\le \min(\delta,\beta-\alpha)$, and consider  ${\mathcal P}_n$ with $n\ge N$. If 
 $ \mathcal P_n$ consists of the  intervals $\{I_{k}^n\}$,  $\mathcal P_n^1$  is the  collection $\{I_{k}^{n,1}\}$, with $I_{k}^{n,1}\subset I_{k}^{n}$, and  $|I_{k}^{n,1}|= \phi(\alpha+|I_k^n|)$.

Observe that for $x_k^{n,1}$ in  $I_k^{n,1}$,  we have,
\begin{align} s(f,\Psi, {\mathcal P}_n^1)
= 
\sum_{k} f(x_k^{n,1}&)\,|I_k^{n,1}|_\Psi\nonumber \\
=\sum_{k} f(x_k^{n,1}&)\int_{I_k^{n,1}}\big(\psi-\psi(x_k^{n,1})\big) +
\sum_{k} f(x_k^{n,1}) \psi(x_k^{n,1})\,|I_k^{n,1}|\nonumber 
\\
&=A_n+B_n,
\end{align} 
 say.
Now, since  
\begin{equation}
\int_{I_k^{n,1}} \big|\psi - \psi(x_k^{n,1})\big| 
\le {{\rm osc}\,(\psi,I_k^{n,1})}\,|I_k^{n,1}|  \le  {{\rm osc}\,(\psi,I_k^n)}\, |I_k^n|, 
\end{equation}
with $M_f$ a bound for $f$, by (6) it follows that, 
\begin{equation}\limsup_n |A_n|\le M_f\limsup_n \sum_k  {{\rm osc}\,(\psi,I_k^n)}\, |I_k^n| =0\,. 
\end{equation}

As for $B_n$, it equals
\begin{align*}
\sum_{k} f(x_k^{n,1})\,\psi(x_k^{n,1})\, \phi(\alpha +&|I_k^n|) 
\nonumber\\
=\sum_{k} f(x_k^{n,1})&\psi(x_k^{n,1})\,\Big( \frac{\phi(\alpha +|I_k^n|)}{|I_k^n|}-\phi'_+(\alpha)\Big)\,|I_k^n|\\
&+\phi'_+(\alpha)\,\sum_{k} f(x_k^{n,1})\psi(x_k^{n,1})\, |I_k^n| =C_n+D_n,
\end{align*}
say. 
Since for the intervals $I_k^n$ in ${\mathcal P}_n$ we have $|I_k^n|\le \delta$,   by (12),   $C_n$ is bounded by
\begin{align*} \Big| \sum_{k} f(&x_k^{n,1})\,\psi(x_k^{n,1})\Big( \frac{\phi(\alpha +|I_k^n|)}{|I_k^n|}-\phi'_+(\alpha)\Big)\,|I_k^n| \Big|\\
&\le   \sum_{k} |\,f(x_k^{n,1})\psi(x_k^{n,1})\,|\,\Big| \frac{\phi(\alpha +|I_k^n|)}{|I_k^n|}-\phi'_+(\alpha)\Big|\,|I_k^n| \le M_\psi\,M_f\, |I|\,\varepsilon,
\end{align*}
which, since $\varepsilon$ is arbitrary,   implies that
\begin{equation}\limsup_n |C_n|=0\,.
\end{equation}
Also, 
\begin{equation}D_n = \phi'_+(\alpha)\,  S(f\,\psi,{\mathcal P}_n).
\end{equation}

Hence, combining (13), (15),  (16), and (17), it follows that
\begin{equation*}
u(f,\Psi)= \lim_n s(f,\Psi,{\mathcal P}_n^1) = \phi'_+(\alpha)\, \lim_n S(f\,\psi,{\mathcal P}_n)= \phi'_+(\alpha) \int_I f\psi,
\end{equation*}
and we have finished.

The choice $\phi(t)=\gamma\, t$ above, and the corresponding  mappings $\Phi_\gamma$, $0<\gamma <1$,  which assign to an interval $J\subset I$,  
$ \Phi_\gamma(J)=J^1\subset I$, where  $\chi_{J^1}$ is Riemann integrable on $I$, and   $J^1$ is  of relative length $\gamma$ in $J$ whenever $|J|<\eta$, apply to  the signal retrieval result described in the introduction. As anticipated, in this case the modified Riemann sums of $f$ with respect $\Psi$ on $I$ converge  to\,  $\gamma \int_I f\,\psi$.

Along similar lines, with $\alpha \in  (\, 0, |I|\,]$, let $\phi$ be a non-negative function defined on $[\,\alpha-|I|, \alpha\,]$  that satisfies: (i) $\phi(t)\le \alpha-t$, for $t$ in a left-neighbourhood of $\alpha$, (ii) $\phi(\alpha)=0$, and, (iii) $\phi$ is left-differentiable at $\alpha$. Then,
define $\Phi(J)=J^1$, where $|J^1|=\phi(\alpha-|J|)$. Note that $\phi(\alpha -| J|)\le \alpha -(\alpha - |J|)= |J|$,  for $|J|$ small. Then, for these values we may, and do pick $J^1=\Phi(J)\subset J$  with $\chi_{\Phi(J)}$  Riemann integrable. The reader will have no difficulty in proving that 
in this case, analogously to (17),   we have
\[u(f,\Psi)= \phi'_-(\alpha) \int_I f\,\psi.
\]

We will consider next the case when the mapping $\Phi$ depends on the location of the intervals $J$.

\begin{theoremc} Let  $\lambda$ be a continuous function such that $\lambda >1$ on $I$, and let  $0<\gamma \le 1$. Then, for an interval $J=[c,d]\subset I$, let $\Phi(J)=J^1$ where $J^1=[c,c']$ is such that $\lambda(c')\int_{[c,c']}\psi=\gamma \int_I \psi\,$; since $c'<d$, $J^1\subset J$. 

Then, if $f\lambda$ is integrable with respect to $\Psi$ on $I$, $f$ is  integrable with respect to $\Psi$ on $I$, and 
\[u(f\lambda,\Psi)=\gamma \int_I f\,d\Psi.
\]
\end{theoremc}

The proof of this result is left to the interested reader.

Along similar lines, we have,

\begin{theoremd} Let $\lambda$ be a bounded, Riemann integrable function on $I$ such that $\lambda> 1$ on $I$, and let $\Upsilon$ be an indefinite integral of $\lambda\,\psi$. Fix $0<\gamma\le 1$, and for each interval $J=[c,d]\subset I $, pick $c'\in I$ such that $\int_{[c,c']} \lambda\,\psi =\gamma \int_J \psi.$ Let $\Phi$ be defined by $\Phi(J)=J^1=[c,c']\subset J$. 

Let $f$ be integrable with respect to $\Upsilon$ on $I$. Then, 
$f$ is integrable with respect to $\Psi$ on $I$, and
\[u( f,\Upsilon)=  \gamma \int_I f\,d\Psi.\]
\end{theoremd}

First, since $\Upsilon(x)=\Upsilon(a)+ \int_{[a,x]} \lambda\,\psi, \, x\in I,$  for $[c,d]\subset I$ we have, 
\[\Psi(d)-\Psi(c)=\int_{[c,d]} \psi\le \int_{[c,d]}\lambda \,\psi =\Upsilon(d)-\Upsilon(c),
\]
and, therefore, if $f$ is integrable with respect to $\Upsilon$ on $I$, by (5), $f$ is integrable with respect to $\Psi$ on $I$. 

Then, for a sequence of partitions $\{\mathcal P_n\}$ of $I$ that satisfies simultaneously (3), and hence also (11),  for $f$ with respect to $\Upsilon$, and (3), and hence (4),  for $f$ with repect to $\Psi$,  and  (6) for $\lambda$,  with $\mathcal P_n=\{I_k^n\}$, we have $I_k^{n,1}=[x_{k,l},x_{k,l}']$, we have for $x_k^{n,1}\in I_k^{n,1},$ 
\begin{align*} s(f,\Upsilon,\mathcal{P}^1_n)&= \sum_{k} f(x_k^{n,1})\, |I^1_k|_\Upsilon=\sum_{k} f(x_k^{n,1})\, \int_{I_k^1} \lambda\,\psi \\
&= \sum_{k} f(x_k^{n,1})\,\gamma \int_{I_k} \psi  =\gamma\,  S(f, \Psi, \mathcal P_n)\,, 
\end{align*}
and, so, taking limits,
\[u( f,\Upsilon)= s( f,\Upsilon)= \gamma \int_I f\,d\Psi.\]

This expression, along with the  example that follows, reflect a change of variable for the Riemann-Stieltjes integral;  general results in this direction can be found in \cite{Torch}.

\begin{theoreme} Given a non-decreasing, Lipschitz function $\Lambda$ on $I$ with Lipschitz constant $1$, let   $\Phi(J)=J^1\subset I$, where, if $J=[c,d]$,   $|J^1|= \Lambda(d)-\Lambda(c)$. 

Then, if $f$ is integrable with respect to $\Psi$ on $I$, $f\psi$ is integrable with respect to $\Lambda$ on $I$, and
\[u(f,\Psi)=\int_I f\psi\, d\Lambda. \]
\end{theoreme}

First,    since  $(\Lambda(d)-\Lambda(c))\le (d-c)$,  it is possible to pick  $J^1\subset J$  with $\chi_{\Phi(J)}$  Riemann integrable, and we do so. Next, if $f$ be integrable with respect to $\Psi$ on $I$,  $f\psi$ is integrable on $I$,
and since  for  $J=[c,d]$, $(\Lambda(d)-\Lambda(c))\le (d-c)$, by (5), 
$f\psi$ is integrable with respect to $\Lambda$ on $I$.

Let  the sequence $\{\mathcal{P}_n\}$  of partitions of $I$ satisfy (3), and hence (11), for $f$ with respect to $\Psi$, (3), and hence (4), for $f\psi$ with respect to $\Lambda$, and (6) for $\psi$. If $\mathcal P_n$ consists of the intervals  $I_k=[x_{k,l}^n, x_{k,r}^n]$,  and $\mathcal{P}_n^1$ denotes the collection  $\{I_k^{n,1}\}$,  we have 
\begin{equation*} |I_k^{n,1}|_\Psi= \int_{I_k^{n,1}}\psi
=\int_{I_k^{n,1}} (\psi-\psi(x_k^{n,1})) + \psi(x_k^{n,1})\big( \Lambda(x_{k,r}^n) -\Lambda(x_{k,l}^n)\big),
\end{equation*}
where the first summand is bounded by
\[ {{\rm osc}\,(\psi, I_k^{n,1})} \big( \Lambda(x_{k,r}^n) -\Lambda(x_{k,l}^n)\big)\le  {{\rm osc}\,(\psi, I_k^n)}\,|I_k^n|,\quad {{\rm all}}\ k, n\,.  
\]

Therefore,
\begin{align*} s(f,\Psi,\mathcal{P}^1_n) &=
 \sum_{k} f(x_k^{n,1})\, |I^{n,1}_k|_\Psi
 \\
 &= \sum_{k} \int_{I_k^{n,1}} f(x_k^{n,1}) (\psi-\psi(x_k^{n,1}))+ S(f\psi,\Lambda,\mathcal{P}_n),
\end{align*}
where by (6) the first sum is bounded by
 \[ M_f \sum_{k} {{\rm osc}\,(\psi, I_k^n)}\,|I_k^n|\to 0,\quad {{\rm as}}\ n\to \infty.\]

Hence,  
\[u(f,\Psi)=s(f,\Psi)= \lim_n s(f,\Psi,\mathcal{P}^1_n)= \lim_n S(f\psi,\Lambda,\mathcal{P}_n)=\int_I f\psi\, d\Lambda. \]

\end{document}